\documentclass[10pt]{elsart}
\usepackage{amssymb}
\usepackage{amsfonts}
\usepackage{amsmath}

\input{tcilatex}

\begin{document}

%
\begin{frontmatter}

%

\title{$k$-colored kernels}%

\hspace{0.38in}
\author{Hortensia Galena-S\'{a}nchez $^{\mathrm{a}}$, Bernardo Llano $^{\mathrm{b,1}}$, and}
\hspace{0.38in}
\author{Juan Jos\'{e} Montellano-Ballesteros $^{\mathrm{a}}$}
\footnote{Corresponding author. \\ E-mail addresses: hgaleana@matem.unam.mx (H. Galena-S\'{a}nchez), 
                                                                                   llano@xanum.uam.mx (B. Llano),
                                                                                   juancho@matem.unam.mx (J. J. Montellano-Ballesteros). }%

\address{$^{\mathrm{a}}$ Instituto de Matem\'{a}ticas, UNAM, Ciudad Universitaria, 04510, M\'{e}xico, D. F.}
\address{$^{\mathrm{b}}$ Departamento de Matem\'{a}ticas, Universidad Aut\'{o}noma Metropolitana, Iztapalapa, San Rafael Atlixco 186, Colonia Vicentina, 09340, M\'{e}xico, D.F.}%

%
\begin{abstract}

%
We study $k$-colored kernels in $m$-colored digraphs. An $m$-colored digraph 
$D$ has $k$-colored kernel if there exists a subset $K$ of its vertices such
that

(i) from every vertex $v\notin K$ there exists an at most $k$-colored
directed path from $v$ to a vertex of $K$ and

(ii) for every $u,v\in K$ there does not exist an at most $k$-colored
directed path between them.

In this paper, we prove that for every integer $k\geq 2$ there exists a $%
(k+1)$-colored digraph $D$ without $k$-colored kernel and if every directed
cycle of an $m$-colored digraph is monochromatic, then it has a $k$-colored
kernel for every positive integer $k.$ We obtain the following results for
some generalizations of tournaments:

(i) $m$-colored quasi-transitive and $3$-quasi-transitive digraphs have a $k$%
-colored kernel for every $k\geq 3$ and $k\geq 4,$ respectively (we
conjecture that every $m$-colored $l$-quasi-transitive digraph has a $k$%
-colored kernel for every $k\geq l+1)$, and

(ii) $m$-colored locally in-tournament (out-tournament, respectively)
digraphs have a $k$-colored kernel provided that every arc belongs to a
directed cycle and every directed cycle is at most $k$-colored.

\end{abstract}%

\begin{keyword}
$m$-colored digraph, $k$-colored kernel, $k$-colored absorbent set, $k$%
-colored independent set 
\end{keyword}%

\end{frontmatter}%

\section{Introduction}

Let $j,$ $k$ and $m$ be positive integers. A digraph $D$ is said to be $m$%
\textit{-colored} if the arcs of $D$ are colored with $m$ colors. Given $%
u,v\in V(D),$ a directed path from $u$ to $v$ of $D,$ denoted by $u\leadsto
v,$ is $j$\textit{-colored }if all its arcs\textit{\ }use exactly\textit{\ }$%
j$ colors and it is represented by $u\leadsto _{j}v.$ When $j=1,$ the
directed path is said to be \textit{monochromatic. }A nonempty set $%
S\subseteq V(D)$ is a $k$\textit{-colored} \textit{absorbent set }if for
every vertex $u\in V(D)-S$ there exists $v\in S$ such that $u\leadsto _{j}v$
with $1\leq j\leq k.$ A nonempty set $S\subseteq V(D)$ is a called a $k$%
\textit{-colored independent} \textit{set }if for every $u,v\in S$ there
does not exist $u\leadsto _{j}v$ with $1\leq j\leq k.$ Let $D$ be an $m$%
-colored digraph. A set $K\subseteq V(D)$ is called a $k$\textit{-colored
kernel }if $K$ is a $k$-colored absorbent and independent set.

This notion was introduced in \cite{MCh} and it is a natural generalization
to kernels by monochromatic directed paths (the case of $1$-colored kernels)
defined first in \cite{GS1}. It is also a generalization of the classic
notion of kernels in digraphs which has widely studied in different contexts
and has an extensive literature (see \cite{BG} for a recent remarkable
survey on the topic). In \cite{MCh}, among other questions, the following
problem is solved: given an $m$-colored digraph $D,$ is it true that $D$ has
a $k$-colored kernel if and only if $D$ has a $(k+1)$-colored kernel? The
answer is no for both implications. In fact, it can be proved (see Theorems
6.6 and 6.7 of the already mention thesis) that

\begin{enumerate}
\item[(i)] for every couple of positive integers $k$ and $s$ with $k<s,$
there exists a digraph $D$ and an arc coloring of $D$ such that $D$ has a $k$%
-colored kernel and $D$ does not have an $s$-colored kernel and

\item[(ii)] for every couple of positive integers $k$ and $s$ with $k<s,$
there exists a digraph $D$ and an arc coloring of $D$ such that $D$ has an $%
s $-colored kernel and $D$ does not have a $k$-colored kernel.
\end{enumerate}

The study of $1$-colored kernels in $m$-colored digraphs already has a
relatively extense literature and has explored sufficient conditions for the
existence of such kernels in many infinite families of special digraphs as
tournaments (particularly, in connection with so called Erd\H{o}s' problem,
see for instance \cite{Ming}, \cite{SSW}, \cite{GS-RM1} and \cite{GS-RM2})
and its generalizations (multipartite tournaments and quasi-transitive
digraphs). As well, it has been of interest searching coloring conditions on
subdigraphs of general digraphs to guarantee the existence of $1$-colored
kernels. These results inspire this work. Taking into account these kind of
theorems, it is natural to ask which of them can be generalized or could be
weakened in their conditions so we can ensure that there exists a $k$%
-colored kernel. In this paper we begin to take the first steps toward this
goal and exhibit some results of this type.

In Section 3, we show some results for general digraphs. In particular, we
prove that for every integer $k\geq 2$ there exists a $(k+1)$-colored
digraph without $k$-colored kernel and if every directed cycle of an $m$%
-colored digraph is monochromatic, then it has a $k$-colored kernel for
every positive integer $k.$ In Section 4, we show that $m$-colored
quasi-transitive and $3$-quasi-transitive digraphs have $k$-colored kernel
for every $k\geq 3$ and $k\geq 4,$ respectively. These theorems provide
evidences to conjecture that every $m$-colored $l$-quasi-transitive digraph
(see its definition in the preliminaries) has a $k$-colored kernel for every 
$k\geq l+1.$ Finally, it is proved that an $m$-colored locally in-tournament
(out-tournament, respectively) digraph has a $k$-colored kernel if

\begin{enumerate}
\item[(i)] every arc belongs to a directed cycle and

\item[(ii)] every directed cycle is at most $k$-colored.
\end{enumerate}

\section{Preliminaries}

In general, we follow the terminology and notations of \cite{BJ-G}. If $T$
is directed path of $D$ and $w,z\in V(T),$ then $w\leadsto ^{T}z$ denotes
the directed subpath from $w$ to $z$ along $T.$

A subdigraph $D^{\prime }$ of an $m$-colored digraph $D$ is said to be 
\textit{monochromatic }if every arc of $D^{\prime }$ is colored alike and
let \textit{colors}$(D^{\prime })$ denote the set of colors used by the arcs
of $D^{\prime }.$

Recall that a \textit{kernel }$K$ of $D$ is an independent set of vertices
so that for every $u\in V(D)\setminus K$ there exists $(u,v)\in A(D),$ where 
$v\in K.$ We say that a digraph $D$ is \textit{kernel perfect }if every
nonempty induced subdigraph of $D$ has a kernel. An arc $(u,v)\in A(D)$ is 
\textit{asymmetrical }(resp. \textit{symmetrical}) if $(v,u)\notin A(D)$
(resp. $(v,u)\in A(D)$).

Given an $m$-colored digraph $D,$ we define the $k$\textit{-colored closure}
of $D,$ denoted by $\mathfrak{C}_{k}(D),$ as the digraph such that $V(%
\mathfrak{C}_{k}(D))=V(D)$ and%
\begin{equation*}
A(\mathfrak{C}_{k}(D))=\{(u,v):\exists \,u\leadsto _{j}v,1\leq j\leq k\}.
\end{equation*}

\begin{remark}
Observe also that every $m$-colored digraph $D$ has a $k$-colored kernel if
and only if $\mathfrak{C}_{k}(D)$ has a kernel. \label{closure}
\end{remark}

We will use the following theorem of P. Duchet \cite{Duchet}.

\begin{theorem}
If every directed cycle of a digraph $D$ has a symmetrical arc, then $D$ is
kernel-perfect. \label{Duchet}
\end{theorem}

In \cite{Berge} (see Corollary 2 on page 311) it is proved that in a digraph 
$D,$ there exists a set $B\subseteq V(D)$ such that

\begin{enumerate}
\item[(I)] no directed path joins two distinct vertices of $B$ and

\item[(A)] each vertex $u\notin B$ is the initial endpoint of a directed
path finishing in a vertex of $B.$
\end{enumerate}

Let us define a set of vertices $B\subseteq V(D)$ to be a \textit{kernel by
directed paths of }$D,$ if it satisfies conditions (I) (independence by
directed paths) and (A) (absorbency by directed paths). Therefore, the
result stated before can be settled as

\begin{theorem}
Every digraph $D$ has a kernel by directed paths. \label{Berge}
\end{theorem}

A digraph $D$ is called $k$-\textit{quasi-transitive }if whenever distinct
vertices $u_{0},u_{1},\ldots ,u_{k}\in V(D)$ such that 
\begin{equation*}
u_{0}\longrightarrow u_{1}\longrightarrow \cdots \longrightarrow u_{k}
\end{equation*}%
there exists at least $(u_{0},u_{k})\in A(D)$ or $(u_{k},u_{0})\in A(D).$
When $k=2$ or $3,$ the digraph $D$ is simply said to be \textit{%
quasi-transitive }(introduced in \cite{BJ}) or $3$\textit{-quasi-transitive}
(see \cite{BJ-3-qt} and \cite{GS-G-U}) respectively. The following is a
well-known result proved in \cite{BJ}.

\begin{proposition}[\protect\cite{BJ}, Corollary 3.2]
If a quasi-transitive digraph $D$ has a directed path $u\leadsto v$ but $%
(u,v)\notin A(D),$ then either $(v,u)\in A(D),$ or there exist vertices $%
x,y\in V(D)-\{u,v\}$ such that 
\begin{eqnarray*}
u &\longrightarrow &x\longrightarrow y\longrightarrow v\text{ and} \\
v &\longrightarrow &x\longrightarrow y\longrightarrow u
\end{eqnarray*}%
are directed paths in $D.$ \label{CorBJ}
\end{proposition}

The next result is a special case of a more general statement proved in \cite%
{GC-HC} ($d(u,v)$ denotes the distance from $u$ to $v$ for $u,v\in V(D)$).

\begin{proposition}[\protect\cite{GC-HC}, Lemma 4.4]
Let $D$ be a $3$-quasi-transitive digraph and $u,v\in V(D)$ such that there
exists $u\leadsto v.$ Then

\begin{enumerate}
\item[(i)] if $d(u,v)=3$ or $d(u,v)\geq 5,$ then $d(v,u)=1$ (that is $%
(v,u)\in A(D))$ and

\item[(ii)] if $d(u,v)=4,$ then $d(v,u)\leq 4$. \label{lemma-HC}
\end{enumerate}
\end{proposition}

A\textit{\ tournament} $T$ on $n$ vertices is an orientation of a complete
graph $K_{n}.$ A digraph $D$ is a \textit{locally in-tournament digraph}
(resp. \textit{locally out-tournament digraph}) if for every $u\in V(D)$ the
set $N^{-}(u)$ of in-neighbors of $u$ (resp. $N^{+}(u)$ of out-neighbors of $%
u$) induces a tournament.

\section{Results for general digraphs}

The research on $1$-colored kernels (or kernels by monochromatic directed
paths) goes back to the classical result of Sands, Sauer and Woodrow (see 
\cite{SSW}) who proved that every $2$-colored digraph has a $1$-colored
kernel. A generalization of this theorem for $k$-colored kernels when $k\geq
2$ is not longer true as the following result shows.

\begin{theorem}
For every $k\geq 2$ there exists a $(k+1)$-colored digraph $D$ without $k$%
-colored kernel.
\end{theorem}

\begin{proof}
Let $D$ be a directed cycle $(u_{0},u_{1},u_{2},\ldots ,u_{2k+1},u_{0})$
such that arcs $(u_{2i},u_{2i+1})$ and $(u_{2i+1},u_{2i+2})$ are colored
with color $i$ for every $0\leq i\leq k$ (the subindices of the vertices are
taken modulo $2k+2).$ First, notice that if for some $0\leq m\leq 2k+1,$ a
vertex $u_{m}\in V(D)$ belongs to a $k$-colored kernel $K,$ then $%
K=\{u_{m}\}.$ Otherwise, if $u_{l}\in K$ for some $0\leq l\leq 2k+1,$ $l\neq
m,$ then there exists either $u_{m}\leadsto _{j}u_{l}$ or $u_{l}\leadsto
_{j}u_{m}$ with $1\leq j\leq k+1$ and $K$ is not $k$-colored independent.
Now, we prove that the singleton set $K$ is not even $k$-colored absorbent.
By the symmetry of the arc coloring of $D,$ we can assume without loss of
generality that $m=1$ or $m=2.$ If $m=1,$ then there exists $u_{2}\leadsto
_{k+1}u_{1}$ but there does not exist $u_{2}\leadsto _{k}u_{1}$ and $K$ is
not $k$-colored absorbent. Analogously, if $m=2,$ there exists $%
u_{3}\leadsto _{k+1}u_{2}$ but there does not exist $u_{3}\leadsto _{k}u_{2}$
and $K$ is not $k$-colored absorbent.
\end{proof}

Let $k$ be a positive integer. As a direct consequence of Theorem \ref{Berge}%
, we have the following

\begin{theorem}
Let $D$ be an $m$-colored digraph satisfying that for every $u,v\in V(D),$
if there exists $u\leadsto v,$ then there exists $u\leadsto _{j}v$ with $%
1\leq j\leq k.$ Then $D$ has a $k^{\prime }$-colored kernel for every $%
k^{\prime }\geq k.$
\end{theorem}

\begin{corollary}
Let $D$ be an $m$-colored digraph satisfying that for every $u,v\in V(D),$
if there exists $u\leadsto v,$ then there exists $u\leadsto v$ of length at
most $k.$ Then for every arc coloring of $D$ with $m$ colors $(m\geq 1),$
the digraph $D$ has a $k$-colored kernel for every $k\geq 1.$
\end{corollary}

Let us denote by $\limfunc{diam}(D)$ the \textit{diameter }of a digraph $D$
(defined to be $\max \{d(u,v):u,v\in V(D)\}$ where $d(u,v)$ is the distance
from $u$ to $v).$

\begin{corollary}
Let $D$ be an $m$-colored strongly connected digraph such that $diam(D)\leq
k.$ The $D$ has a $k^{\prime }$-colored kernel for every $k^{\prime }\geq k.$
\end{corollary}

The following lemma will be an useful tool in proving Theorem \ref{thm-mono}.

\begin{lemma}
Let $D$ be an $m$-colored digraph such that every directed cycle of $D$ is
monochromatic. Let $u,v\in V(D)$ and $T_{1}$ denotes the directed path $%
u\leadsto v.$ If $T_{2}$ is a $v\leadsto u,$ then colors$(T_{1})=c$\textit{%
olors}$(T_{2}).$ \label{colors}
\end{lemma}

\begin{proof}
We proceed by induction on the length of $T_{1},$ denoted by $l(T_{1}).$ If $%
l(T_{1})=1,$ then $T_{1}\cup T_{2}$ is a directed cycle which is
monochromatic by supposition, so \textit{colors}$(T_{1})=c$\textit{olors}$%
(T_{2}).$ Suppose that the lemma is valid for every $l(T_{1})\leq n$ for
some positive integer $n$ and we prove the claim for $l(T_{1})=n+1.$ We have
two cases:

\textbf{Case 1. }$V(T_{1})\cap V(T_{2})=\{u,v\}.$ Then $T_{1}\cup T_{2}$ is
a directed cycle which is monochromatic and consequently, \textit{colors}$%
(T_{1})=c$\textit{olors}$(T_{2}).$

\textbf{Case 2. }$\left\vert V(T_{1})\cap V(T_{2})\right\vert \geq 3.$ Let $%
z\in V(T_{1})\cap V(T_{2})$ be the first vertex of $T_{2}$ in $T_{1}$. Then 
\begin{equation*}
\gamma =(z\leadsto ^{T_{1}}v)\cup (v\leadsto ^{T_{2}}z)
\end{equation*}
is a directed monochromatic cycle. Let $T_{1}^{\prime }$ and $T_{2}^{\prime
} $ be $u\leadsto ^{T_{1}}z$ and $z\leadsto ^{T_{2}}u$ respectively. Since $%
l(T_{1}^{\prime })\leq n,$ by induction hypothesis, we have that \textit{%
colors}$(T_{1}^{\prime })=c$\textit{olors}$(T_{2}^{\prime }).$ Without loss
of generality, suppose that $\gamma $ is of color $1.$

\textbf{Subcase 2.1.} Color $1$ does not appear in $T_{1}^{\prime }.$ Since 
\textit{colors}$(T_{1}^{\prime })=c$\textit{olors}$(T_{2}^{\prime }),$ then
color $1$ does not appear in $T_{2}^{\prime }$ either and therefore%
\begin{equation*}
\mathit{colors}(T_{1})=c\mathit{olors}(T_{1}^{\prime })\cup \{1\}=c\mathit{%
olors}(T_{2}^{\prime })\cup \{1\}=c\mathit{olors}(T_{2}).
\end{equation*}

\textbf{Subcase 2.2.} Color $1$ appears in $T_{1}^{\prime }.$ Since \textit{%
colors}$(T_{1}^{\prime })=c$\textit{olors}$(T_{2}^{\prime }),$ then color
also appears in $T_{2}^{\prime }$ and therefore%
\begin{equation*}
\mathit{colors}(T_{1})=c\mathit{olors}(T_{1}^{\prime })=c\mathit{olors}%
(T_{2}^{\prime })=c\mathit{olors}(T_{2}).
\end{equation*}
\end{proof}

\begin{theorem}
Let $D$ be an $m$-colored digraph such that every directed cycle of $D$ is
monochromatic. Then $D$ has a $k$-colored kernel for every $k\geq 1.$ \label%
{thm-mono}
\end{theorem}

\begin{proof}
In virtue of Theorem \ref{Duchet} and Remark \ref{closure}, we will prove
that every directed cycle of $\mathfrak{C}_{k}(D)$ has a symmetrical arc.
So, let $\gamma =(u_{0},u_{1},\ldots ,u_{n-1},u_{0})$ be a directed cycle of 
$\mathfrak{C}_{k}(D),$ where $n$ is a positive integer. Then for every $%
0\leq i\leq n-1$ there exists $u_{i}\leadsto _{j}u_{i+1}$ in $D,$ denoted by 
$T_{i},$ where $1\leq j\leq k$ and the indices are taken modulo $n.$
Therefore, $\bigcup_{i=1}^{n-1}T_{i}$ is a directed walk from $u_{1}$ to $%
u_{0}$ and accordingly there exists an $u_{1}\leadsto u_{0}$ which is
denoted by $\alpha .$ Since $T_{0}$ is a $u_{0}\leadsto u_{1},$ by Lemma \ref%
{colors}, \textit{colors}$(T_{0})=c$\textit{olors}$(\alpha ).$ Using that $%
T_{0}$ is $j$-colored with $1\leq j\leq k,$ we conclude that $\alpha $ is a $%
u_{1}\leadsto _{j}u_{0}$ with $1\leq j\leq k.$ Thus, $%
(u_{0},u_{1}),(u_{1},u_{0})\in A(\mathfrak{C}_{k}(D))$ and $\gamma $ has a
symmetrical arc.
\end{proof}

In case of $k=1$ in Theorem \ref{thm-mono} we have the following

\begin{corollary}[\protect\cite{GS-GG-RM}]
Let $D$ be an $m$-colored digraph such that every directed cycle of $D$ is
monochromatic. Then $D$ has a kernel by monochromatic directed paths.
\end{corollary}

\section{$k$-colored kernels in some generalizations of tournaments}

Let $T$ be a tournament. It is well-known that each vertex of maximum
out-degree of a tournament $T$ is a king (see Subsection 3.7.1 of \cite{BJ-G}%
). Dually, if $x$ is vertex of maximum in-degree in $T,$ then $y\leadsto
_{j}x$ with $1\leq j\leq 2$ for every $y\in V(T)$. Therefore $\{x\}$ is a $2$%
-colored absorbent set and so $\{x\}$ is a $2$-colored kernel of $T.$
Observe that a $k$-colored absorbent set is also $l$-colored absorbent for
every positive integer $l>k$ an we conclude that a tournament $T$ has a $k$%
-colored kernel for every $k\geq 2.$

\subsection{Quasi-transitive digraphs}

A similar result can be proved for quasi-transitive digraphs when $k\geq 3.$

\begin{lemma}
Let $D$ be an $m$-colored quasi-transitive digraph. If there exists $%
u\leadsto _{j}v$ with $1\leq j\leq k$ and there exists no $v\leadsto _{j}u$
with $1\leq j\leq k,$ then $(u,v)\in A(D).$ \label{lemma-qt}
\end{lemma}

\begin{proof}
Consider $u\leadsto _{j}v$ with $1\leq j\leq k$ of minimum length. Using
Proposition \ref{CorBJ}, we have one of the following possibilities:

\begin{enumerate}
\item[(i)] $(v,u)\in A(D)$ which is impossible by supposition (there exists
no $v\leadsto _{j}u$ with $1\leq j\leq k).$

\item[(ii)] $(u,v)\in A(D)$ and the conclusion follows.

\item[(iii)] There exist vertices $x,y\in V(D)-\{u,v\}$ such that 
\begin{eqnarray*}
u &\longrightarrow &x\longrightarrow y\longrightarrow v\text{ and} \\
v &\longrightarrow &x\longrightarrow y\longrightarrow u
\end{eqnarray*}
are directed paths in $D.$ This is also impossible since 
\begin{equation*}
v\longrightarrow x\longrightarrow y\longrightarrow u
\end{equation*}
is a $v\leadsto _{j}u$ with $1\leq j\leq 3.$
\end{enumerate}
\end{proof}

\begin{theorem}
Let $D$ be an $m$-colored quasi-transitive digraph. Then $D$ has a $k$%
-colored kernel for every $k\geq 3.$
\end{theorem}

\begin{proof}
We will prove that every directed cycle in $\mathfrak{C}_{k}(D)$ has a
symmetrical arc. By contradiction, suppose that there exists a directed
cycle $\gamma =(u_{0},u_{1},\ldots ,u_{n-1},u_{0})$ without symmetrical arcs
belonging to $\mathfrak{C}_{k}(D).$ By the definition of the $k$-colored
closure, there exists $u_{i}\leadsto _{j}u_{i+1}$ with $1\leq j\leq k$ and
there exists no $u_{i+1}\leadsto _{j}u_{i}$ with $1\leq j\leq k$ in $D$ for
every $0\leq i\leq n-1$ (indices are taken modulo $n).$ By Lemma \ref%
{lemma-qt}, we have that $(u_{i},u_{i+1})\in A(D)$ for every $0\leq i\leq
n-1.$ Therefore $\gamma $ is a directed cycle of $D.$ On the other hand, $%
\gamma $ has at least one change of color in its arcs, otherwise $\gamma $
would be monochromatic and then $u_{i+1}\leadsto _{j}u_{i}$ with $1\leq
j\leq k$ and for every $0\leq i\leq n-1.$ Without loss of generality, we can
assume that the color change occurs at vertex $u_{1}$ and so arcs $%
(u_{0},u_{1})$ and $(u_{1},u_{2})$ are colored, say, of colors $1$ and $2$
respectively. Since $D$ is quasi-transitive, there exists an arc between $%
u_{0}$ and $u_{2}.$ We consider two cases:

\textbf{Case 1. }$(u_{2},u_{0})\in A(D).$ Then there exists the directed
path $u_{1}\longrightarrow u_{2}\longrightarrow u_{0}$ which is at most $2$%
-colored, a contradiction since there is no $u_{1}\leadsto _{j}u_{0}$ with $%
1\leq j\leq k$ and $k\geq 3.$

\textbf{Case 2. }$(u_{0},u_{2})\in A(D).$ Since $D$ is quasi-transitive,
there exists an arc between $u_{0}$ and $u_{3}.$ If $(u_{3},u_{0})\in A(D),$
then we have a contradiction similar to Case 1 with vertices $u_{0},$ $u_{2}$
and $u_{3}.$ So, $(u_{0},u_{3})\in A(D)$ and the procedure continues. Since $%
(u_{n-1},u_{0})\in A(D),$ there exists a first index $i$ such that $1\leq
i\leq n-1$ and $(u_{i},u_{0})\in A(D).$ Then $(u_{0},u_{i-1})\in A(D).$ But
then there exists an at most $2$-colored directed path $u_{i}\longrightarrow
u_{0}\longrightarrow u_{i-1}$ and this is a contradiction to the fact that
there is no $u_{i}\leadsto _{j}u_{i-1}$ with $1\leq j\leq k$ in $D.$
\end{proof}

We remark that an analogous argument as before can be used to prove that an $%
m$-colored quasi-transitive digraph $D$ such that every directed triangle is
monochromatic has a $k$-colored kernel for $k=1,2.$ We do not know if for
the case $k=2,$ the monochromacity of the directed triangles can be weakened.

\subsection{3-quasi-transitive digraphs}

A more elaborated proof allows us to show that a $3$-quasi-transitive
digraph has a $k$-colored kernel for every $k\geq 4.$

Let us define the \textit{flower }$F_{r}$ with $r$ petals as the digraph
obtained by replacing every edge of the star $K_{1,r}$ by a symmetrical arc.
If every edge of the complete graph $K_{n}$ is replaced by a symmetrical
arc, then the resulting digraph $D$ on $n$ vertices is symmetrical
semicomplete.

\begin{remark}
Let $F_{r}$ be an $m$-colored flower such that $r\geq 1.$ Then $\mathfrak{C}%
_{k}(F_{r})$ with $k\geq 2$ is a symmetrical semicomplete digraph. \label%
{flower}
\end{remark}

\begin{theorem}
Let $D$ be an $m$-colored $3$-quasi-transitive digraph. Then $D$ has a $k$%
-colored kernel for every $k\geq 4.$
\end{theorem}

\begin{proof}
Applying Theorem \ref{Duchet} and Remark \ref{closure}, we will prove that
every directed cycle of $\mathfrak{C}_{k}(D)$ has a symmetrical arc. By
contradiction, suppose that $\gamma =(u_{0},u_{1},\ldots ,u_{p},u_{0})$ is a
cycle in $\mathfrak{C}_{k}(D)$ without any symmetrical arc. Observe that if $%
p=1,$ then $\gamma $ has a symmetrical arc and we are done. So, assume that $%
p\geq 2.$

Using Proposition \ref{lemma-HC}, if there exists an arc $(u_{i},u_{i+1})\in
A(\mathfrak{C}_{k}(D))$ with $0\leq i\leq p$ (indices are taken modulo $p+1)$
of $\gamma $ which corresponds to a directed path of length at least $3$ in $%
D,$ then there exists an $u_{i+1}\leadsto u_{i}$ of length at most $4$. This
is a contradiction, $\gamma $ has a symmetrical arc between $u_{i}$ and $%
u_{i+1}.$ So, we assume that every arc of $\gamma $ corresponds to an arc or
a directed path of length $2$ in $D.$

Let $\delta $ be the closed directed walk defined by the concatenation of
the arcs and the directed paths of length $2$ corresponding to the arcs of $%
\gamma .$

First, observe that $\delta $ contains a directed path of length at least $%
3, $ otherwise $\delta $ is isomorphic to $F_{r}$ with $r\geq 2$ and by
Remark \ref{flower}, $\gamma $ has a symmetrical arc, a contradiction.

If $\delta $ contains a flower $F_{r},$ $r\geq 1$ such that two consecutive
vertices $u_{i}$ and $u_{i+1}$ of $\gamma $ belong to the vertices of $%
F_{r}, $ then by Remark \ref{flower}, there exists a symmetrical arc between 
$u_{i}$ and $u_{i+1}.$ Therefore we can assume that 
\begin{equation}
\text{there are no consecutive vertices of }\gamma \text{ in a flower.} 
\tag{*}
\end{equation}

Observe that if $\delta =\gamma ,$ as we will see, the same argument of the
proof will work even easier.

Let $\delta =(y_{0},y_{1},...,y_{s}).$ Observe that there exist $%
y_{i_{0}},y_{i_{1}},...y_{i_{p}}\in V(\delta )$ such that $i_{j}<i_{j+1}$
and $u_{l}=y_{j_{l}},$ where $0\leq l\leq p.$

We define $\varepsilon =(y_{i},y_{i+1},\ldots ,y_{i+l})$ of minimum length ($%
0\leq i\leq s$ and the indices are taken modulo $s+1)$ such that

\begin{enumerate}
\item[(i)] $y_{i}=y_{i+l},$ $l\geq 3,$

\item[(ii)] $y_{i}\neq y_{t}$ for $i+1\leq t\leq i+l-1,$

\item[(iii)] if $y_{q}=y_{r},$ then $q=r+2$ ($i+1\leq q,r\leq i+l-1$)$,$

\item[(iv)] there exist $y_{i_{1}},y_{i_{2}},\ldots ,y_{i_{k+1}}\in
V(\varepsilon )$ such that $y_{i_{1}}=u_{j},$ $y_{i_{2}}=u_{j+1},$ $\ldots ,$
$y_{i_{k+1}}=u_{j+k}$ with $k\geq 1,$ and

\item[(v)] $y_{i+1}\neq y_{i+l-1}.$
\end{enumerate}

\begin{claim}
There exists $\varepsilon $ subdigraph of $\delta .$
\end{claim}

Since $\delta $ is a closed walk, $p\geq 2$ and using (*), condition (i) is
satisfied. If there exists $t<l$ such that $y_{i}=y_{i+t},$ then, by the
minimality of $\varepsilon ,$ $t=2$ and $l-t=2$ and therefore $(i+l)-(i+t)=2 
$. By (i), we have that $y_{i}=y_{i+t}=y_{i+l}$ and so $l=4$. We obtain that 
\begin{equation*}
y_{i+1}\longleftrightarrow y_{i}=y_{i+2}=y_{i+4}\longleftrightarrow y_{i+3}.
\end{equation*}

Using that every arc of $\gamma $ corresponds to an arc or a directed path
of length $2,$ there exist two consecutive vertices of $\gamma $ in the
closed walk 
\begin{equation*}
y_{i}\longrightarrow y_{i+1}\longrightarrow y_{i+2}\longrightarrow
y_{i+3}\longrightarrow y_{i+4}.
\end{equation*}%
Therefore, we have two consecutive vertices of $\gamma $ in a flower of two
petals, a contradiction to the assumption (*). Condition (iii) follows from
the minimality of $\varepsilon $ and (iv) is immediate from the definition
of $\delta $ and the fact that $l\geq 3.$ If $y_{i+1}=y_{i+l-1},$ then by
(iii), $l=4$ and therefore%
\begin{equation*}
y_{i+4}=y_{i}\longleftrightarrow y_{i+1}=y_{i+3}\longleftrightarrow y_{i+2}
\end{equation*}%
which is a flower with two petals and again there are two consecutive
vertices of $\gamma $ in a flower of two petals, a contradiction to the
assumption (*). Condition (v) follows.

Claim 1 is proved.

We remark that as a consequence of this claim and since $\delta $ is not a
flower by supposition, $\varepsilon $ is a directed cycle of length at least 
$3$ with perhaps symmetrical arcs attached to some vertices (maybe none) of
the cycle for which the exterior endpoints are vertices of $\gamma $. For
example, closed walk $\varepsilon $ with a directed cycle of length $6$ and
two symmetrical arcs attached is depicted in the next figure. Observe that
in this example, $y_{i+1},$ $y_{i+5}$ and $y_{i+9}$ are elements of $\gamma $
by the definition of $\delta .$%
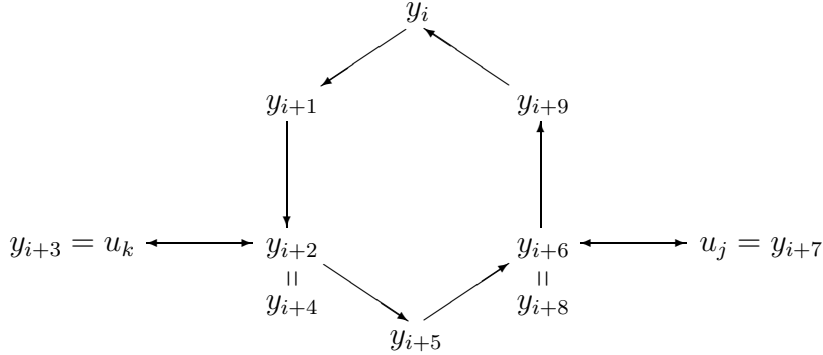
\begin{equation*}
\begin{picture}(250,145)
\put(82,40){\vector(3,-2){32}}    \put(107,10) {$y_{i+5}$}    \put(120,19){\vector(3,2){32}}
\put(-37,45) {$y_{i+3}=u_k$}    \put(15,48){\vector(1,0){40}}    \put(55,48){\vector(-1,0){40}}    
\put(60,23) {\shortstack {$y_{i+2}$\\ $\shortparallel$ \\ $y_{i+4}$}} 
\put(155,23)  {\shortstack {$y_{i+6}$\\ $\shortparallel$ \\ $y_{i+8}$}} 
\put(179,48){\vector(1,0){40}}    \put(219,48){\vector(-1,0){40}}    
\put(224,45) {$u_j=y_{i+7}$}
\put(68,94){\vector(0,-1){40}}      \put(164,54){\vector(0,1){40}} 
\put(60,99) {$y_{i+1}$}     \put(155,99) {$y_{i+9}$}
\put(113,129){\vector(-3,-2){32}}      \put(152,108){\vector(-3,2){32}}
\put(113,134) {$y_i$}
\end{picture}%
\end{equation*}

Let us rename $\varepsilon =(y_{0},y_{1},\ldots ,y_{l}).$ By (v) of the
definition of $\varepsilon $, we have that $y_{1}\neq y_{l-1}$ and by (iv),
there exist consecutive $u_{0},u_{1},\ldots ,u_{k}\in V(\gamma )$ in $%
\varepsilon $ with $k\geq 1.$ Notice that $u_{0}$ and $u_{k}$ could not be
consecutive vertices of $\gamma $ and similarly, $(y_{l-1},y_{0})\in
A(\varepsilon )$ could not be an arc of $\gamma .$ Let $u_{1}=y_{i}$ be the
second vertex of $\gamma $ from $y_{0}.$ Observe that $1\leq i\leq 3$ by the
definition of $\varepsilon $ and either $u_{0}=y_{0}$ $(1\leq i\leq 2)$ or $%
u_{0}=y_{1}$ $(2\leq i\leq 3).$

\begin{claim}
$(u_{1},y_{j})\in A(D)$ for some $l-2\leq j\leq l-1.$
\end{claim}

To prove the claim, let $q$ be the maximum index such that $(u_{1},y_{q})\in
A(D)$ and $q\leq l-3.$ Consider $y_{q}\longrightarrow y_{q+1}\longrightarrow
y_{q+2}$. Observe that by the maximality of $q,$ $y_{q+2}\neq y_{q}$. Then%
\begin{equation*}
u_{1}\longrightarrow y_{q}\longrightarrow y_{q+1}\longrightarrow y_{q+2}
\end{equation*}%
is a directed path and $q+2\leq l-1.$ Since $D$ is $3$-quasi-transitive and
using the maximality of $q,$ we have that $(y_{q+2},u_{1})\in A(D).$ There
are no consecutive vertices $u_{i}$ and $u_{i+1}$ of $\gamma $ in $%
\{y_{q},y_{q+1},y_{q+2}\},$ otherwise there exists a $u_{i+1}\leadsto u_{i}$
of length at most $4$ and therefore $\gamma $ has a symmetrical arc between $%
u_{i}$ and $u_{i+1}$, a contradiction. Thus, the only possibility is that $%
y_{q+1}=u_{t}$ for some $2\leq t\leq k,$ $y_{q+3}\in V(\gamma )$ and $%
y_{q+3}\neq u_{t}.$ There exists the directed path%
\begin{equation*}
u_{1}\longrightarrow y_{q}\longrightarrow y_{q+1}=u_{t}\longrightarrow
y_{q+2}\longrightarrow y_{q+3}
\end{equation*}%
such that $y_{q+3}\neq u_{0}$, otherwise there exists $u_{1}\leadsto u_{0}$
of length at most $4$ and therefore $\gamma $ has a symmetrical arc between $%
u_{0}$ and $u_{1}$, a contradiction. Therefore $y_{q+3}=u_{t+1}$ with $%
t+1\leq k.$ Since $D$ is $3$-quasi-transitive, there exists an arc between $%
y_{q}$ and $y_{q+3}.$ If $(y_{q+3},y_{q})\in A(D),$ then there exists%
\begin{equation*}
y_{q+3}=u_{t+1}\longrightarrow y_{q}\longrightarrow y_{q+1}=u_{t}
\end{equation*}%
of length $2$ and $\gamma $ has a symmetrical arc between $u_{t}$ and $%
u_{t+1}$, a contradiction. So $(y_{q},y_{q+3})\in A(D).$ Consider the vertex 
$y_{q+4}$ (observe that $q+3\leq l-1).$ We have that $y_{q+4}\neq y_{q+2},$
otherwise 
\begin{equation*}
u_{t+1}=y_{q+3}\longrightarrow y_{q+4}=y_{q+2}\longrightarrow
u_{1}\longrightarrow y_{q}\longrightarrow y_{q+1}=u_{t}
\end{equation*}%
is a directed path of length $4$ and $\gamma $ has a symmetrical arc between 
$u_{t}$ and $u_{t+1}$, a contradiction. We notice that $y_{q+4}\notin
\{u_{0},y_{0}=y_{l}\},$ otherwise there exists either the directed path%
\begin{eqnarray*}
u_{1} &\longrightarrow &y_{q}\longrightarrow y_{q+3}=u_{t+1}\longrightarrow
y_{q+4}=y_{0}=u_{0}\text{ or} \\
u_{1} &\longrightarrow &y_{q}\longrightarrow y_{q+3}=u_{t+1}\longrightarrow
y_{q+4}=y_{0}\longrightarrow y_{1}=u_{0}.
\end{eqnarray*}%
Then, there exists $u_{1}\leadsto u_{0}$ of length at most $4$ and thus $%
\gamma $ has a symmetrical arc between $u_{0}$ and $u_{1}$, a contradiction.
Therefore $q+4\leq l-1$ and $y_{q+4}\neq y_{q}$ by the maximality of $q.$
Since there exists the directed path 
\begin{equation*}
u_{1}\longrightarrow y_{q}\longrightarrow y_{q+3}=u_{t+1}\longrightarrow
y_{q+4},
\end{equation*}%
we have that there is an arc between $u_{1}$ and $y_{q+4}$ (recall that $D$
is $3$-quasi-transitive). By the maximality of $q,$ we obtain that $%
(y_{q+4},u_{1})\in A(D),$ but in this case we have the directed path%
\begin{equation*}
u_{t+1}=y_{q+3}\longrightarrow y_{q+4}\longrightarrow u_{1}\longrightarrow
y_{q}\longrightarrow y_{q+1}=u_{t}
\end{equation*}%
and there exists $u_{t+1}\leadsto u_{t}$ of length $4$ and thus $\gamma $
has a symmetrical arc between $u_{t}$ and $u_{t+1}$, a contradiction. The
procedure of the claim is illustrated in the following figure.%
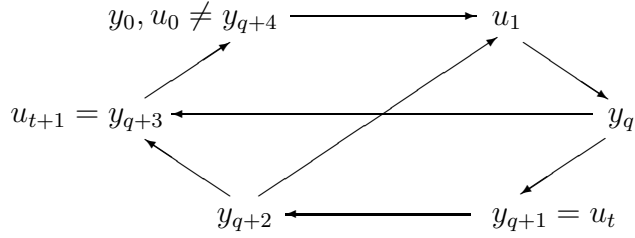
\begin{equation*}
\begin{picture}(250,90)
\put(101,10) {\vector(3,2){90}}
\put(44,75) {$y_0,u_0 \neq y_{q+4}$}   \put(113,78) {\vector(1,0){70}}  \put(189,75) {$u_1$}
\put(58,47) {\vector(3,2){32}}   \put(201,68) {\vector(3,-2){32}}
\put(7,38) {$u_{t+1}=y_{q+3}$}   \put(227,41) {\vector(-1,0){159}}   \put(233,38) {$y_q$}
\put(90,10) {\vector(-3,2){32}}  \put(232,32) {\vector(-3,-2){32}} 
\put(85,0) {$y_{q+2}$}   \put(181,3) {\vector(-1,0){70}}    \put(189,0) {$y_{q+1}=u_t$}
\end{picture}%
\end{equation*}

Claim 2 is proved.

As a consequence of this claim and since either $u_{0}=y_{0}$ or $%
u_{0}=y_{1} $, we have the directed path $u_{1}\longrightarrow y_{j}\leadsto
u_{0}$ $(l-2\leq j\leq l-1)$ of length at most $4$ and $\gamma $ has a
symmetrical arc between $u_{0}$ and $u_{1}$, a contradiction concluding the
proof.
\end{proof}

The results obtained for quasi-transitive and $3$-quasi-transitive digraphs
suggest that

\begin{conjecture}
Let $D$ be an $m$-colored $l$-quasi-transitive digraph. Then $D$ has a $k$%
-colored kernel for every $k\geq l+1.$
\end{conjecture}

\subsection{Locally in- and out-tournament digraphs}

\begin{lemma}
Let $D$ be an $m$-colored locally out-tournament digraph such that

\begin{enumerate}
\item[(i)] every arc belongs to a directed cycle and

\item[(ii)] every directed cycle is at most $k$-colored.
\end{enumerate}

If there exists an $u\leadsto v,$ then there exists a $v\leadsto _{j}u$ with 
$1\leq j\leq k.$ \label{local-out}
\end{lemma}

\begin{proof}
Let 
\begin{equation*}
\beta :u=u_{0}\longrightarrow u_{1}\longrightarrow \cdots \longrightarrow
u_{n}=v
\end{equation*}%
be a directed path from $u$ to $v$ of minimum length. Let $u_{j}\in V(D)$ be
such that $1\leq j\leq n$ is the maximum index such that $u_{0}$ and $u_{j}$
belong to a same cycle 
\begin{equation*}
\gamma =(u_{0}=w_{0},w_{1},\ldots ,w_{l-1},w_{l}=u_{j},w_{l+1},\ldots
,w_{m}=u_{0})
\end{equation*}%
where $m\geq 1.$ If $j=n,$ we are done since every cycle of $D$ is at most $%
k $-colored. Let us suppose that $j<n$ and consider the vertex $u_{j+1}\in
V(\beta ).$ Since $D$ is locally out-tournament, $(u_{j},u_{j+1})\in A(D)$
and $(u_{j},w_{l+1})\in A(D)$, there exists an arc between $u_{j+1}$ and $%
w_{l+1}.$ If $(u_{j+1},w_{l+1})\in A(D)$, then there exists a directed cycle%
\begin{equation*}
(u_{0}=w_{0},w_{1},\ldots ,w_{l-1},w_{l}=u_{j},u_{j+1},w_{l+1},\ldots
,w_{m}=u_{0})
\end{equation*}%
which contradicts the maximality of the choice of $j.$ Therefore there
exists $(w_{l+1},u_{j+1})\in A(D).$ Since $D$ is locally out-tournament,
there exists an arc between $u_{j+1}$ and $w_{l+2}.$ If $(u_{j+1},w_{l+2})%
\in A(D)$, then there exists a directed cycle%
\begin{equation*}
(u_{0}=w_{0},w_{1},\ldots ,w_{l}=u_{j},w_{l+1},u_{j+1},w_{l+2},\ldots
,w_{m}=u_{0})
\end{equation*}%
which contradicts the maximality of the choice of $j.$ Therefore there
exists $(w_{l+2},u_{j+1})\in A(D).$ The procedure continues until we obtain
that there exists $(w_{m-1},u_{j+1})\in A(D).$ Again, since $D$ is locally
out-tournament, there exists an arc between $u_{j+1}$ and $u_{0}.$ If there
exists $(u_{j+1},u_{0})\in A(D)$, then there exists a directed cycle%
\begin{equation*}
(u_{0}=w_{0},w_{1},\ldots ,w_{l}=u_{j},u_{j+1},u_{0}),
\end{equation*}%
a contradiction to the maximality of $j.$ If there exists $%
(u_{0},u_{j+1})\in A(D)$, then the directed path 
\begin{equation*}
u=u_{0}\longrightarrow u_{j+1}\longrightarrow u_{j+2}\longrightarrow \cdots
\longrightarrow u_{n}=v
\end{equation*}%
is of shorter length than $\beta ,$ a contradiction to the choice of $\beta
. $
\end{proof}

\begin{lemma}
Let $D$ be an $m$-colored locally in-tournament digraph such that

\begin{enumerate}
\item[(i)] every arc belongs to a directed cycle and

\item[(ii)] every directed cycle is at most $k$-colored.
\end{enumerate}

If there exists an $u\leadsto v,$ then there exists a $v\leadsto _{j}u$ with 
$1\leq j\leq k.$ \label{local-in}
\end{lemma}

\begin{proof}
Let 
\begin{equation*}
\beta :u=u_{0}\longrightarrow u_{1}\longrightarrow \cdots \longrightarrow
u_{n}=v
\end{equation*}%
be a directed path from $u$ to $v$ of minimum length. Let $u_{j}\in V(D)$ be
such that $0\leq j\leq n-1$ is the minimum index such that $u_{j}$ and $%
u_{n} $ belong to a same cycle 
\begin{equation*}
\gamma =(u_{n}=w_{0},w_{1},\ldots ,w_{l-1},w_{l}=u_{j},w_{l+1},\ldots
,w_{m}=u_{n})
\end{equation*}%
where $m\geq 1.$ If $j=0,$ we are done since every cycle of $D$ is at most $%
k $-colored. Let us suppose that $j>0$ and consider the vertex $u_{j-1}\in
V(\beta ).$ Since $D$ is locally in-tournament, $(u_{j-1},u_{j})\in A(D)$
and $(w_{l-1},u_{j})\in A(D)$, there exists an arc between $u_{j-1}$ and $%
w_{l-1}.$ If $(w_{l-1},u_{j-1})\in A(D)$, then there exists a directed cycle%
\begin{equation*}
(u_{n}=w_{0},w_{1},\ldots ,w_{l-1},u_{j-1},w_{l}=u_{j},w_{l+1},\ldots
,w_{m}=u_{n})
\end{equation*}%
which contradicts the minimality of the choice of $j.$ Therefore there
exists $(u_{j-1},w_{l-1})\in A(D).$ Since $D$ is locally in-tournament,
there exists an arc between $u_{j-1}$ and $w_{l-2}.$ If $(w_{l-2},u_{j-1})%
\in A(D)$, then there exists a directed cycle%
\begin{equation*}
(u_{n}=w_{0},w_{1},\ldots ,w_{l-2},u_{j-1},w_{l-1},w_{l}=u_{j},\ldots
,w_{m}=u_{n})
\end{equation*}%
which contradicts the minimality of the choice of $j.$ Therefore there
exists $(u_{j-1},w_{l-2})\in A(D).$ The procedure continues until we obtain
that there exists $(u_{j-1},w_{1})\in A(D).$ Again, since $D$ is locally
in-tournament, there exists an arc between $u_{j-1}$ and $u_{n}.$ If there
exists $(u_{n},u_{j-1})\in A(D)$, then there exists a directed cycle%
\begin{equation*}
(u_{n}=w_{0},u_{j-1},w_{l}=u_{j},w_{l+1},\ldots ,w_{m}=u_{n}),
\end{equation*}%
a contradiction to the minimality of $j.$ If there exists $%
(u_{j-1},u_{n})\in A(D)$, then the directed path 
\begin{equation*}
u_{0}\longrightarrow u_{1}\longrightarrow \cdots u_{j-1}\longrightarrow
u_{n}=v
\end{equation*}%
is of shorter length than $\beta ,$ a contradiction to the choice of $\beta
. $
\end{proof}

\begin{theorem}
Let $D$ be an $m$-colored locally out-tournament (in-tournament,
respectively) digraph such that

\begin{enumerate}
\item[(i)] every arc belongs to a directed cycle and

\item[(ii)] every directed cycle is at most $k$-colored.
\end{enumerate}

Then $D$ has a $k$-colored kernel. \label{at-most-k}
\end{theorem}

\begin{proof}
By Lemma \ref{local-out} (resp. Lemma \ref{local-in}), every arc of $%
\mathfrak{C}_{k}(D)$ is symmetrical and therefore $\mathfrak{C}_{k}(D)$ has
a kernel in virtue of Theorem \ref{Duchet}. By Remark \ref{closure}, we
conclude that $D$ has a $k$-colored kernel.
\end{proof}

In view of Theorems \ref{thm-mono} and \ref{at-most-k}, it is natural to ask
about the minimum number of colors of a directed cycle in a digraph $D$ such
that $D$ has a $k$-colored kernel. We state the following

\begin{conjecture}
Let $D$ be an $m$ colored digraph such that every arc belongs to a directed
cycle and every directed cycle is at most $k$-colored. Then $D$ has a $k$%
-colored kernel.
\end{conjecture}

\end{document}